\documentclass[a4paper,reqno]{amsart}

\usepackage{amssymb,graphicx,xcolor}
\usepackage{amssymb,bbm,enumerate}
\usepackage{esint}
\usepackage{colortbl}
\usepackage{xcolor}

\numberwithin{equation}{section}\newtheorem{theorem}{Theorem}[section]
\newtheorem{corollary}[theorem]{Corollary}\newtheorem{lemma}[theorem]{Lemma}

\theoremstyle{remark}

\theoremstyle{definition}

\newcommand{\supp}{\mathrm{supp\,}}
\newcommand{\R}{\mathbb{R}}
\newcommand{\Rn}{\mathbb{R}^{n}}

\title
[A necessary condition for the Schr\"odinger maximal estimate]{An improved necessary condition for the Schr\"odinger maximal estimate}
\date{\today}    %%% ''\date{}'' to omit date

\author{Renato Luc\`a}
\author{Keith M. Rogers}
\address{Instituto de Ciencias Matem\'aticas CSIC-UAM-UC3M-UCM, Madrid, 28049, Spain.}
\email{renato.luca@icmat.es, keith.rogers@icmat.es}


\begin{document}
\begin{abstract}   We improve the necessary condition for Carleson's problem regarding convergence for the Schr\"odinger equation in dimensions $n\ge 3$. We prove that if the solution  converges almost everywhere to its initial datum as time tends to zero, for all data in $H^s(\R^n)$, then $s\ge \frac{n}{2(n+2)}$.
\end{abstract}

\maketitle

\section{Introduction}

Consider the  Schr\"odinger equation, $i \partial_{t} u +\Delta u=0$, on $\R^{n+1}$, with initial datum 
$u(\cdot, 0) = u_{0}$, and 
 Carleson's problem of identifying 
the exponents $s > 0$ for which
\begin{equation}\label{CarlesonProblem}
\lim_{t\to 0} u(x , t) = u_{0}(x), \qquad \text{a.e.} \quad x \in \Rn, \qquad \forall \  u_0\in H^s.
\end{equation}
Here, $H^s$ denotes the inhomogeneous $L^2(\R^n)$--Sobolev space, defined via the Fourier transform as usual.
Carleson \cite{Carl} proved that \eqref{CarlesonProblem} holds as long as
 $s \geq 1/4$ in the one-dimensional case, and   Dahlberg and Kenig \cite{DahlKenig}  showed that this condition is necessary  in all dimensions, providing a complete solution for the one-dimensional case. 
The higher dimensional problem  has since been studied by many authors; see for example
\cite{Cow, Carb,  Sjolin, Vega, Bou3, Bou4, MVV1, MVV2, TV, T, GS}. The best known positive result, that (\ref{CarlesonProblem}) holds if 
$$s > \frac{1}{2}-\frac{1}{4n},$$ is due to Lee \cite{Lee} when $n=2$ and Bourgain~\cite{Bou2} when $n\ge 3$. Bourgain \cite{Bou2} also showed that $s\ge 1/2-1/n$ is necessary for \eqref{CarlesonProblem} to hold, improving the condition of Dahlberg and Kenig when $n\ge 5$.

Here we improve Bourgain's necessary condition and the condition of  Dahlberg and Kenig when $n\ge 3$.

\begin{corollary}\label{ones}
Let  $n\ge 3$ and suppose that \eqref{CarlesonProblem} holds. Then $s\ge\frac{1}{2}-\frac{1}{n+2}$.
\end{corollary}

When  the initial data $u_{0}$ is a Schwartz function, we can write 
\begin{equation}\label{rt}
u(x, t) = e^{i t \Delta}u_{0}(x):=\int_{\Rn}\widehat{u}_{0}(\xi)\, e^{2\pi ix\cdot\xi -4\pi^2it | \xi |^{2}} d \xi, 
\end{equation}
where $\widehat{u}_0$ denotes the Fourier transform of $u_0$.
By the Niki\v sin--Stein maximal
principle \cite{N, st}, the almost everywhere
convergence \eqref{CarlesonProblem} implies a weak $L^2$-estimate  for the
maximal operator, which in turn implies a strong estimate by interpolation with a trivial bound (see for example \cite[Proof of Lemma C.1]{BBCR}). Thus Corollary~\ref{ones} is a consequence of the following theorem.

\begin{theorem}\label{OurBouThm}
Let $n\ge 3$ and suppose that there is a constant $C_s$ such that  
\begin{equation}\label{oto}
\left\| \sup_{0< t < 1} \big| e^{it\Delta} f \big| \right\|_{L^{2}(B(0,1))} \le C_s\| f \|_{H^{s}(\Rn)}
\end{equation}
whenever $f$ is a Schwartz function. Then $s\ge\frac{n}{2(n+2)}$.
\end{theorem}

By the equivalence between local and global estimates \cite{R}, this also yields the following necessary condition for the global maximal estimate.

\begin{corollary}
Let $n\ge 3$ and suppose that there is a constant $C_s$ such that  
\begin{equation*}\label{otoGlobal}
\left\| \sup_{0< t < 1} \big| e^{it\Delta} f \big| \right\|_{L^{2}(\R^n)} \le C_s\| f \|_{H^{s}(\Rn)}
\end{equation*}
whenever $f$ is a Schwartz function. Then $s\ge\frac{n}{n+2}$.
\end{corollary}

The counterexample of Dahlberg and Kenig consists of a concentrated solution, or wave-packet, that travels over a large area, making the left-hand side of \eqref{oto} large. On the other hand, Bourgain considered a sum of data, with different velocities, carefully chosen to create regions of constructive interference, recalling Young's double slit experiment with many slits. Again the regions of coherence travel over a large area, making the left-hand side of the maximal inequality large. 

In the light of Bourgain's example, a physical interpretation of Carleson's problem could be to identify the lowest frequency at which an initial state (or configuration of slits)  can generate interference patterns, thus obscuring their original state. 
%It seems reasonable to suppose that in the absence of interference the initial state can be recovered by tracing back from the later states.
Inspired by this, we take a variant of data, previously considered by Barcel\'o, Bennett, Carbery, Ruiz and Vilela \cite{BBCRV}, for which the corresponding solution interferes with itself periodically in time. The difficulty of using  their example directly in this context is that the constructive interference reoccurs in the same relatively small regions of space. In order to take advantage of the periodic coherence, we perturb the initial state so that the whole solution travels in a single direction.  We then use an ergodicity argument to show that this direction can be taken so that the regions of constructive interference never reappear in exactly the same places, forcing the left-hand side of \eqref{oto} to be large.

\section{The ergodic lemma}

We say that a set $E$ is $\delta$--dense in $F$ if for every point $x\in F$ there is a point $y\in E$ such that $|x-y|< \delta$. 

The following lemma is optimal in the sense that the statement fails for larger~$\sigma$. To see this, we can place balls of radius $\varepsilon R^{-1}$ at the points of the set $E_\theta$ and assume that the balls are disjoint. Then the volume of such a set would be of the order $R^{1-(n+2)\sigma}$, a quantity that tends to zero as $R$ tends to infinity when $\sigma > \frac{1}{n+2}$.

\begin{lemma}\label{Lemma:ToroImpr}
Let $n \geq 3$ and $0<\sigma <\frac{1}{n+2}$. Then for all $\varepsilon>0$, there exists  $\theta \in \mathbb{S}^{n-1}$ such that
\begin{equation*}
 E_\theta:= \bigcup_{t\in R^{2\sigma-1}\mathbb{Z}\cap(0,1)} \big\{x\in R^{\sigma-1} \mathbb{Z}^{n}\, :\, |x|< 2\big\}+t\theta
\end{equation*}
is $\varepsilon R^{-1}$--dense in $B(0,1/2)$ for all sufficiently large $R >1$. 
\end{lemma}

\begin{proof}
By rescaling, the statement of the lemma  is equivalent to showing that
\begin{equation}\nonumber
\bigcup_{t \in R^{\sigma} \mathbb{Z}\cap(0,R^{1-\sigma})} \big\{x\in \mathbb{Z}^{n}\, :\, |x|< 2R^{1-\sigma} \big\}+t\theta
\end{equation}
is $\varepsilon R^{-\sigma}$--dense in $B(0,R^{1-\sigma}/2)$ for a certain $\theta \in \mathbb{S}^{n-1}$.
That is to say, for any $x\in B(0,R^{1-\sigma}/2)$ there exists a $y_{x} \in \mathbb{Z}^{n}\cap B(0,2R^{1-\sigma})$ and
$t_{x}\in R^{\sigma} \mathbb{Z}\cap (0,R^{1-\sigma})$ such that
\begin{equation*}\label{EquivTorusImpr}
|x - (y_{x} + t_{x}\theta) |< \varepsilon R^{-\sigma},
\end{equation*}
for a certain $\theta \in \mathbb{S}^{n-1}$, independent of $x$.
By taking the quotient $\Rn / \mathbb{Z}^{n} = \mathbb{T}^{n}$, this would follow if for any $[x] \in \mathbb{T}^{n}$ there exists $t_{x} \in R^{\sigma} \mathbb{Z}\cap (0,R^{1-\sigma})$  such that
\begin{equation}\label{TorusRiductionImpr}
|[x] - [t_{x} \theta]| < \varepsilon R^{-\sigma}.
\end{equation} 
To see this, assume  (\ref{TorusRiductionImpr}) and cover
 $B(0,R^{1-\sigma}/2)$ with a family of disjoint copies of axis-parallel $\mathbb{T}^{n}$. Denote 
the copy that contains $x$ by $\mathbb{T}^{n}_{x}$,
and let $z_{x}$ be the point in $\mathbb{T}^{n}_{x}$ such that $[z_{x}] = [t_{x} \theta]$. 
Then $y_{x}: = z_{x} - t_{x}\theta\in  \mathbb{Z}^{n}$ and
 by construction
\begin{equation}
|x- (y_{x}+t_{x} \theta)|  
= |[x] - [t_{x}\theta]| < \varepsilon R^{-\sigma}.
\end{equation} 
Note that we also automatically have that
\begin{equation*}
|y_x|\le |x|+|t_x|+  \varepsilon R^{-\sigma}< \tfrac{1}{2}R^{1-\sigma}+R^{1-\sigma}+\varepsilon R^{-\sigma}<2R^{1-\sigma},
\end{equation*} 
and so we recover all of the required properties.
 It seems likely that ergodic results, similar to   
(\ref{TorusRiductionImpr}), are well-known, however we prove this now using Fourier series.
% as in~\cite{Bou4}. 
We write
$x$ in place of $[x]$ from now on.

Let $\phi : \mathbb{T}^{n} \to [0,(2/\varepsilon)^n)$ be smooth, supported in $B(0,\varepsilon/2)$, such that $\int \phi=1$,  and set 
\begin{equation*}\label{Def:PhiTorusImpr}
\phi_{R}(x) := \phi \big( R^{\sigma}x \big).
\end{equation*}
If we could show that there exists $\theta \in \mathbb{S}^{n-1}$ such that 
for all $x \in \mathbb{T}^{n}$ there is a
$t_{x} \in (R^{\sigma} \mathbb{Z}+[-\tfrac{\varepsilon}{2}R^{-\sigma},\frac{\varepsilon}{2}R^{-\sigma}])\cap (0,R^{1-\sigma})$ satisfying
\begin{equation}\label{TorusRiductionBisImpr}
\phi_{R}(x-t_{x} \theta) > 0, 
\end{equation}
then (\ref{TorusRiductionImpr}) would follow.
Let $\psi: (-\varepsilon/2,\varepsilon/2) \to [0,2/\varepsilon)$ be a one-dimensional Schwartz function such that $\int \psi=1$, 
and define 
$$
\eta_{R}(t) := R^{3\sigma-1} \sum_{\substack{j \in \mathbb{Z} \\ 0 < j < R^{1-2\sigma}}} \psi(R^{\sigma}(t - R^{\sigma} j)).
$$
Noting that $\eta_{R}$ is supported in 
$R^{\sigma} \mathbb{Z}+[-\frac{\varepsilon}{2}R^{-\sigma},\frac{\varepsilon}{2}R^{-\sigma}]$, we will show that there exists $\theta \in \mathbb{S}^{n-1}$ such that,
for all $x \in \mathbb{T}^{n}$,
\begin{equation*}
\int_{\R} \phi_{R}(x-t \theta)  \eta_{R}(t)\, d t  >0,
\end{equation*}
which implies (\ref{TorusRiductionBisImpr}).
Expanding in Fourier series;
\begin{equation}\nonumber
\phi_{R}(x-t\theta) = 
\widehat{\phi_{R}} (0) 
+
 \sum_{\substack{k \in \mathbb{Z}^{n} \\ k\neq 0}} \widehat{\phi_{R}} (k) e^{ 2\pi ix \cdot k}e^{-2\pi i t \theta \cdot k}
=: \widehat{\phi_{R}} (0)+ \Gamma(t, x, \theta), 
\end{equation}
and noting that $\int_{\R} \eta_{R} \simeq 1$ and
$\widehat{\phi_{R}}(0) = \int_{\mathbb{T}^{n}} \phi_{R} \simeq R^{-n\sigma}$,
it would be sufficient 
to find $\theta \in \mathbb{S}^{n}$ such that\footnote{We write $A \lesssim B$ if $ A\leq C B$ for some constant $C > 0$ that only depends on unimportant parameters. We also write $A \simeq B$ if  $ A \lesssim B$ and $B \lesssim A$.}
\begin{equation}\label{WBS}
\Big|  \int_{\R} \Gamma(t, x, \theta)\eta_{R}(t)\,dt  \Big| 
% \leq c\widehat{\phi_{R}} (0)=c\int_{\mathbb{T}^{n}} \phi_{R}
\lesssim R^{-\gamma},\quad \gamma>n\sigma
\end{equation}
whenever $x \in \mathbb{T}^{n}$.

For the proof of \eqref{WBS}, we note that
\begin{eqnarray}\nonumber
\Big|  \int_{\R} \Gamma(t, x, \theta) \eta_{R}(t) dt  \Big|
& \leq & 
\sum_{\substack{k \in \mathbb{Z}^{n} \\ k\neq 0}} \Big| \widehat{\phi_{R}} (k) \Big|
 \ \Big|   \int_{\R} e^{-2\pi i t \theta \cdot k} \eta_{R}(t) d t  \Big|
 \\ \nonumber
 &= &
\sum_{\substack{k \in \mathbb{Z}^{n} \\ k\neq 0}} \Big| \widehat{\phi_{R}} (k) \Big|
 \ \Big|    \widehat{\eta_{R}}( \theta\cdot k)  \Big|\\
&\lesssim& \sum_{\substack{k \in \mathbb{Z}^{n} \\ k\neq 0}} \frac{R^{-n\sigma}}{\left( 1 + R^{-\sigma} |k|  \right)^{n+1}} \Big|    \widehat{\eta_{R}}( \theta\cdot k)  \Big|, \label{pull}
 \end{eqnarray}
where the final inequality uses the Schwartz decay which follows by integrating by parts in the formula for the Fourier coefficients. Noting that the right-hand side of \eqref{pull} no longer depends on~$x$, in order to find a $\theta\in \mathbb{S}^{n-1}$ such that (\ref{WBS}) holds for all $x\in \mathbb{T}^{n}$, it will suffice to prove that the the right-hand side of \eqref{pull} is similarly bounded after averaging over the sphere. As
$$
 \sum_{\substack{k \in \mathbb{Z}^{n} \\ k\neq 0}} 
 \frac{R^{-n\sigma}}{\left( 1 + R^{-\sigma} |k|  \right)^{n+1}}\lesssim  \int_{\Rn} \frac{R^{-n\sigma}}{\left( 1 + R^{-\sigma} |k|  \right)^{n+1}}\, dk\lesssim 1, 
$$
by Fubini's theorem, it would suffice to prove that 
\begin{equation}\label{ORTSINTEGRAL}
\int_{\mathbb{S}^{n-1}} \Big|   \widehat{\eta_{R}}( \theta\cdot k)  \Big|\, d \theta \lesssim  R^{2\sigma-1} \log R.
\end{equation}
We then use that $\sigma<\frac{1}{n+2}$ so that $1-2\sigma>n\sigma$.

To see \eqref{ORTSINTEGRAL},
we calculate
\begin{eqnarray*}\label{ExplFour}
\widehat{\eta_{R}}(t)
&= & 
R^{3\sigma-1} \sum_{\substack{j \in \mathbb{Z} \\ 0 < j < R^{1-2\sigma}}}
\psi\big(R^{\sigma} ( \cdot - R^{\sigma} j)\big)^{\wedge}(t)
\\ \nonumber
& = & 
R^{2\sigma-1} 
\widehat{\psi}(R^{-\sigma} t)
\sum_{\substack{j \in \mathbb{Z} \\ 0 < j < R^{1-2\sigma}}}
 e^{-2\pi i  R^{\sigma}j t }
 \\ \nonumber
& = &
R^{2\sigma-1} 
\widehat{\psi}(R^{-\sigma} t)\frac{e^{2\pi i\lfloor R^{1-2\sigma} \rfloor R^\sigma t}-e^{-2\pi i  R^{\sigma} t }}{e^{2\pi iR^\sigma t}-1}.
\end{eqnarray*}
Now since $| \widehat{\psi} \,| \lesssim 1$ this yields
\begin{equation}\nonumber
  \int_{\mathbb{S}^{n-1}} \Big| \widehat{\eta_{R}}( \theta \cdot k) \Big|\, d \theta 
\lesssim
R^{2 \sigma-1} \int_{\mathbb{S}^{n-1}} 
      \Big|  \frac{\sin(\pi N R^{\sigma} \theta \cdot k)}{\sin (\pi  R^{\sigma} \theta \cdot k)}  \Big| \,  d\theta,
\end{equation}
where $N=\lfloor R^{1-2\sigma} \rfloor+1$.
By the Funk--Hecke theorem (see for example \cite[pp. 35-36]{at}), we have that
\begin{eqnarray*}
\int_{\mathbb{S}^{n-1}} 
      \Big|  \frac{\sin(\pi N R^{\sigma} \theta \cdot k)}{\sin (\pi  R^{\sigma} \theta \cdot k)}  \Big|  \, d\theta&=& |\mathbb{S}^{n-2}| \int_{-1}^1 \Big|  \frac{\sin(\pi NR^{\sigma}|k|t )}{\sin (\pi  R^{\sigma} |k|t)}  \Big| (1-t^2)^{\frac{n-3}{2}} dt\\
      &\leq& \frac{|\mathbb{S}^{n-2}|}{R^{\sigma}|k|} \int_{-R^{\sigma}|k|}^{R^{\sigma}|k|} \Big|  \frac{\sin(\pi Nt )}{\sin (\pi t)}  \Big|\,  dt\\
      &\lesssim& \log N \ \lesssim \ \log R,        
\end{eqnarray*}
where the penultimate inequality is a well-known property of the Dirichlet kernel (see for example \cite[pp. 182]{G}).  This completes the proof of 
(\ref{ORTSINTEGRAL}) which completes the proof of the lemma. 
\end{proof}

\section{Proof of Theorem~\ref{OurBouThm}}
\label{counter}

The maximal estimate \eqref{oto} implies the same estimate over a smaller time interval, and so writing $t/(2\pi R)$ in place of $t$, we know that
\begin{equation}\label{otooBis}
\left\| \sup_{0< t < 1} \big| e^{i\frac{t}{2\pi R}\Delta} f \big| \right\|_{L^{2}(B(0,1))} \lesssim R^s\| f \|_{2}
\end{equation}
whenever $\supp \widehat{f} \subset B(0,2R)$ and $R>1$. Thus it would suffice to prove that for this to hold it is necessary that 
$s \geq \frac{n}{2(n+2)}$. In fact \eqref{otooBis} is equivalent to \eqref{oto}; see~\cite{Lee, LR}, and so we have not thrown anything away here. 

Letting 
$0 < \sigma < \frac{1}{n+2}$
 we define 
\begin{equation}\nonumber
\Omega := \big\{ \xi\in R^{1-\sigma} \mathbb{Z}^{n}  \,:\, |\xi|< R \big\} + B(0,\rho),
\end{equation}
where $\rho$ is to be chosen later.
Let $\theta \in \mathbb{S}^{n-1}$, and consider initial data $f_\theta$ defined by 
%$$\widehat{f_\theta}(\xi)=\widehat{f}\Big(\xi-\frac{R}{2}\theta\Big),$$
\begin{equation*}\label{TestinfFunctBenn}
f_\theta(x) = e^{i\pi R\theta\cdot x} f(x),\quad \text{where}\quad \widehat{f} = \frac{1}{ \sqrt{|\Omega|} }\chi_{\Omega}.
\end{equation*}
Note that $| \supp \widehat{f_\theta}\,| = |\Omega|  \simeq R^{n\sigma}$, and
$ \| f_\theta \|_{2}= 1$.
In \cite{BBCRV}, it was shown that
\begin{equation}\label{Phase=1}
|e^{i\frac{t}{2\pi R}\Delta} f(x)| \gtrsim \sqrt{|\Omega|}
\quad
\quad \forall \
(x,t) \in \Lambda,
\end{equation}
where, taking $\varepsilon$ sufficiently small, $\Lambda$ is defined by
\begin{equation}\nonumber
\Lambda = \big\{ x\in R^{\sigma-1}\mathbb{Z}^{n}\,:\, |x|< 2\big\}+ B(0,\varepsilon R^{-1}) \times   \big\{ t\in  R^{2\sigma-1}\mathbb{Z}\, :\, 0<t<1\big\}.
\end{equation}
 We provide the 
%short 
proof of this  for completeness. The idea is that the phase in the integrand in \eqref{rt} never strays too far from zero modulo~$2\pi i$, and so the different pieces of the integral, corresponding to different pieces of~$\Omega$, cannot cancel each other out. 
In \cite{BBCRV} they proved that the solution is still large in small intervals of time, however this will suffice for our needs.

We start by showing that 
\begin{equation}\label{I*}
x\cdot \xi \in \mathbb{Z} +B(0,\tfrac{1}{20}),
\end{equation}
provided that $\xi \in \Omega$ and $x \in R^{\sigma-1}\mathbb{Z}^{n}\cap B(0,2) +B(0,\varepsilon R^{-1})$.  To see this, we write
\begin{equation}\nonumber
\xi =R^{1-\sigma} \ell   + v, \qquad \text{where}\quad
\ell \in \mathbb{Z}^{n}, \ \ |\ell| < R^{\sigma}, \ \ |v| < \rho 
\end{equation}
and
\begin{equation}\nonumber
x =   R^{\sigma-1}m + u, \qquad \text{where}\quad
m \in \mathbb{Z}^{n}, \ \  |m| < 2R^{1-\sigma}, \  \ |u| < \varepsilon R^{-1}, 
\end{equation}
so that
\begin{eqnarray*}
x \cdot \xi 
& = &
( R^{\sigma-1}m + u)\cdot( R^{1-\sigma} \ell + v)
\\ \nonumber
& = &
  m \cdot \ell 
+   R^{\sigma-1} m\cdot v
+   R^{1-\sigma} \ell \cdot u
+u\cdot v
\\
\nonumber 
& =: & I_{1}+I_{2}+I_{3}+I_{4}.
\end{eqnarray*}
Since $I_{1} \in \mathbb{Z}$ and
\begin{equation}\nonumber
| I_{2} | <  R^{1-\sigma} 2R^{\sigma-1} \rho = 2\rho,
\quad
| I_{3} | <   R^{1-\sigma}R^{\sigma} \varepsilon R^{-1} =  \varepsilon,
\quad
| I_{4} | < \rho \varepsilon R^{-1},
\end{equation}
 we see that (\ref{I*}) holds by taking  $\rho$ and $\varepsilon$ sufficiently small. On the other hand, we also have that
\begin{equation}\label{II*}
\frac{t}{R} |\xi|^{2} \in  \mathbb{Z} + \left(-\tfrac{1}{20}, \tfrac{1}{20}\right),
\end{equation}
provided that
$
t \in   R^{2 \sigma -1}  \mathbb{Z}   \cap (0,1).
$
To see this, we write
\begin{equation}\nonumber
t =  R^{2\sigma -1}k, \qquad \text{where}
\quad
k \in  \mathbb{Z},
\ \
0<k < R^{1-2\sigma}, 
\end{equation}
so that
\begin{eqnarray*}
\frac{t}{R} |\xi|^{2} 
& = & 
  R^{2(\sigma-1)}k| R^{1-\sigma}\ell + v|^{2}
\\ \nonumber
& = &
  R^{2(\sigma-1)}k\big( R^{2(1-\sigma)} | \ell|^{2}  + |v|^{2} +   2R^{1-\sigma}\ell \cdot v\big)
\\ \nonumber
& =: &
I\! I_{1} + I\! I_{2} + I\! I_{3},\end{eqnarray*}
where
$I\! I_{1}  \in \mathbb{Z}$ while
\begin{equation}\nonumber
| I\! I_{2} | \leq   R^{2(\sigma - 1)}  k |v|^{2}  <  R^{2(\sigma-1)} R^{1-2\sigma} \rho^{2} = \rho^{2}R^{-1}, 
\end{equation}
and
\begin{equation}\nonumber
| I\! I_{3} | \leq     R^{2(\sigma-1)} k  2 R^{1-\sigma} |\ell \cdot v| \leq  2R^{\sigma-1} k |\ell| |v| <
 2R^{\sigma-1} R^{1-2\sigma} R^{\sigma} \rho \leq  2\rho,
\end{equation}
so that (\ref{II*}) is satisfied for sufficiently small $\rho$. Indeed altogether $|\rho|, |\varepsilon|\le \frac{1}{100}$ is sufficient for our purposes.
Now (\ref{I*}) and (\ref{II*})
imply that the phase in 
$$
e^{i\frac{t}{2\pi R}\Delta} f(x)=\frac{1}{\sqrt{|\Omega|}}\int_{\Omega} e^{2\pi ix\cdot\xi -2\pi i\frac{t}{R} | \xi |^{2}} d \xi, 
$$
 is close enough to zero modulo $2\pi i$ as long as $(x,t)\in \Lambda$, yielding  (\ref{Phase=1}).

We now consider 
$\Lambda_{\theta,t} \subset \Rn$ defined by
\begin{equation}\nonumber
\Lambda_{\theta,t} :=
\big\{ x\in R^{\sigma-1}\mathbb{Z}^{n}\,:\, |x|< 2\big\}+ B(t\theta,\varepsilon R^{-1}),
\end{equation}
and note that
\begin{equation}\nonumber
 x\in \Lambda_{\theta,t} \quad \text{and} \quad t\in R^{2\sigma - 1} \mathbb{Z}\cap(0,1) \quad \Rightarrow\quad (x-t \theta,t) \in \Lambda.
\end{equation}
Thus, by (\ref{Phase=1}),
we have that
\begin{equation}\nonumber
\sup_{0< t < 1} \big| e^{i\frac{t}{2\pi R}\Delta} f(x -t\theta )|  \gtrsim \sqrt{|\Omega|}
\quad
\quad
\forall \
x \in \Lambda_{\theta} :=\bigcup_{t\in R^{2\sigma - 1} \mathbb{Z}\cap(0,1)}\Lambda_{\theta,t}. 
\end{equation}
 By Galilean invariance, or direct calculation using the formula \eqref{rt}, we have 
$$
\sup_{0< t < 1} \big| e^{i\frac{t}{2\pi R}\Delta} f_\theta(x)|=\sup_{0< t < 1} \big| e^{i\frac{t}{2\pi R}\Delta} f(x -t\theta )|,
$$
and we recall that $\|f_\theta\|_2=\|f\|_{2}= 1$.
Thus,  by  taking $f_\theta$ in (\ref{otooBis}), we obtain
\begin{equation}\nonumber
\sqrt{|\Omega| |\Lambda_{\theta}|} \lesssim R^{s}.
\end{equation}
Since $\Lambda_\theta$ is  nothing more that the $\varepsilon R^{-1}$--neighbourhood of $E_\theta$ from the second section, 
we can use Lemma~\ref{Lemma:ToroImpr} to take  $\theta\in \mathbb{S}^{n-1}$ so that  $|\Lambda_{\theta}| \ge |B(0,1/2)|$  for  sufficiently large $R$. As~$|\Omega| \gtrsim R^{n\sigma}$, we let $R$ tend to infinity so that  
\begin{equation}\nonumber
s \geq \frac{n\sigma}{2},
\end{equation}
and the proof is  completed by letting $\sigma$ tend to  $\frac{1}{n+2}$ as we may.    \hfill $\Box$

\end{document}